\documentclass[12pt]{amsproc}
\usepackage{graphicx, hyperref, xy}
\xyoption{all}
\newtheorem{theorem}{Theorem}
\newtheorem{lemma}[theorem]{Lemma}
\newtheorem{corollary}[theorem]{Corollary}
\newtheorem{definition}[theorem]{Definition}

\newcommand{\odd}{\mathrm{odd}}
\newcommand{\F}{\mathcal F}
\renewcommand{\mod}{\;\mathrm{mod}\;}

\title[Residues in the Young-Fibonacci lattice]{Residues modulo powers of two\\in the Young-Fibonacci lattice}
\author{N.~Karimilla Bi}
\address{The Institute of Mathematical Sciences (HBNI), Chennai}
\email{karimilla.riasm@gmail.com}
\author{Amritanshu Prasad} 
\address{The Institute of Mathematical Sciences (HBNI), Chennai}
\email{amri@imsc.res.in}
\author{P.~Giftson Santhosh}
\address{Ramanujan Institute for Advanced Study in Mathematics, University of Madras, Chennai}
\email{giftsoncbe2@gmail.com}
\begin{document}
\maketitle
\begin{abstract}
  We study the subgraph of the Young-Fibonacci graph induced by elements with odd $f$-statistic (the $f$-statistic of an element $w$ of a differential graded poset is the number of saturated chains from the minimal element of the poset to $w$).
  We show that this subgraph is a binary tree.
  Moreover, the odd residues of the $f$-statistics in a row of this tree equidistibute modulo any power two.
  This is equivalent to a purely number theoretic result about the equidistribution of residues modulo powers of two among the products of distinct odd numbers less than a fixed number.
\end{abstract}
\section{Introduction}
The Young-Fibonacci lattice was introduced by Stanley in \cite{sta-diffpo}.
He presented it as an example of a differential graded poset which is different from Young's lattice.
It shares many combinatorial properties with Young's lattice but, for many purposes, appears to be an easier object to understand.
For instance, the number of elements of rank $n$ in the Young-Fibonacci lattice is given by the $n$th Fibonacci number:
\begin{displaymath}
  F_n = (\phi^n - (-\phi)^{-n})/{\sqrt 5}, \text{ where } \phi = (1+\sqrt 5)/2.
\end{displaymath}
But the number of elements of rank $n$ in Young's lattice is $p_n$, the number of integer partitions of $n$, a function for which no simple closed formula is known to date.

The $f$-statistic of an element of a differential graded poset with minimal element $\emptyset$ is defined recursively by the rules:
\begin{equation}
  \label{eq:1}
  f_\emptyset = 1, \quad f_w = \sum_{v\in w^-} f_v,
\end{equation}
where $w^-$ denotes the set of elements that are covered by $w$.
Many of the enumerative properties of differential posets described in \cite{sta-diffpo} concern the numbers $f_w$.

The Hasse diagram of Young's lattice is known as Young's graph.
The subgraph induced in Young's graph by the partitions $\lambda$ for which $f_\lambda$ is odd was studied in \cite{APS}.
It was shown that this subgraph is a binary tree with a simple recursive structure.
This recursive structure may be viewed as an analog of the well-known relationship between odd entries of Pascal's triangle and the Sierpi\'nski fractal.
Granville \cite{granville1992zaphod} found recursive patterns for odd residue classes modulo powers of two in Pascal's triangle.
However, we were unable to find a recursive description for the residue classes $\pm 1(\mathrm{mod} 4)$ of $f_\lambda$, for $\lambda$ in Young's graph.
There is no obvious pattern for the number of times the residues $+1$ and $-1$ modulo $4$ occur in the rows of Young's graph.

The Hasse diagram of the Young-Fibonacci lattice is called the Young-Fibonacci graph.
In this article, we show that the subgraph induced in the Young-Fibonacci graph  by elements $w$ with $f_w$ odd is again a binary tree with a very simple structure (Theorems~\ref{theorem:macdonald-tree} and~\ref{sec:macdonald-tree-recursive}).
Also, the corresponding $f$-numbers in this tree follow a simple pattern (Theorem~\ref{theorem:f-val}).
Moreover, we show that the odd residue classes of the numbers $f_w$ modulo any power of $2$ eventually equidistribute as $w$ runs over a row of the Young-Fibonacci graph (Theorem~\ref{theorem:main}).
This result turns out to be equivalent to a purely number-theoretic result about the residues of products of distinct odd numbers modulo powers of two (Corollary~\ref{corollary:number-theory}).

\section{The Young-Fibonacci Graph}
Let $F(n)$ denote the set of all words of the form $w=x_1\dotsc x_l$ such that $x_i$ is either $1$ or $2$, and $\sum x_i = n$.
Let $F = \coprod_{n\geq 0} F(n)$.
Note that $F(0)$ has one element, namely the empty word, denoted $\emptyset$.
If $w\in F(n)$ we say that $w$ has rank $n$, and write $r(w)=n$.

The Young-Fibonacci graph (see Figure~\ref{yf}) is the graph with vertex set $F$, and an edge between $v\in F(n)$ and $w\in F(n+1)$ if one of the following holds:
\begin{enumerate}
\item $v$ is obtained from $w$ by changing a $2$ that has no $1$ to its left into a $1$.
\item $v$ is obtained from $w$ by removing its leftmost $1$.
\end{enumerate}
In this case we write $v\in w^-$, or $w\in v^+$.
\begin{figure}[htp]
  \centering
  \hspace{-0.5cm}
  \includegraphics[width=1.2\textwidth]{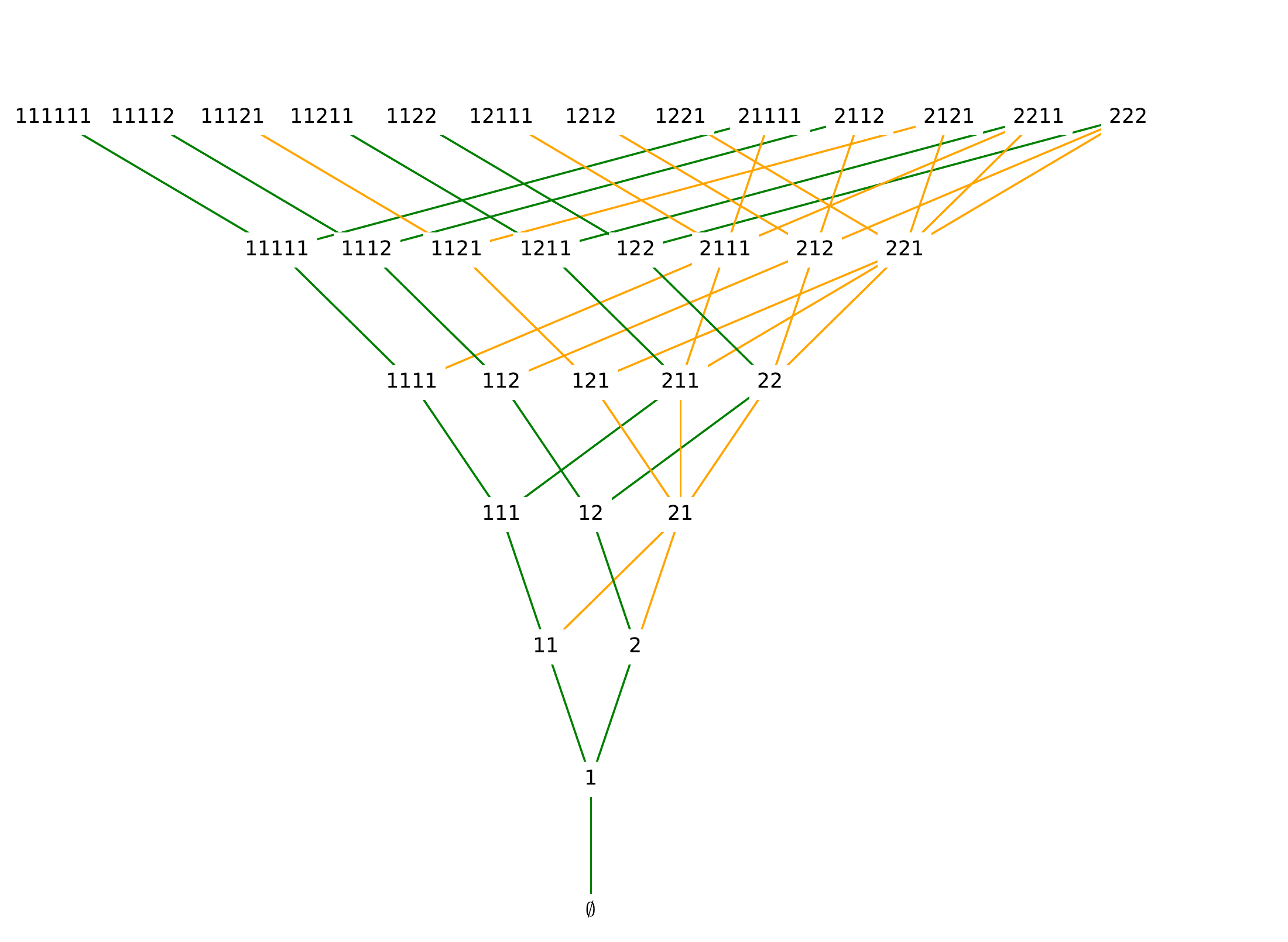}  
  \caption{The Young-Fibonacci Graph}
  \label{yf}
\end{figure}
The set $F$ can be viewed as a partially ordered set in which $w$ covers $v$ if and only if $w\in v^+$.
This poset is known as the Young-Fibonacci lattice.

The following result (see \cite[Proposition~2.3]{nzeutchap}) is an analog of the hook-length formula for the Young-Fibonacci lattice:
\begin{lemma}
  Consider a word $w$ of the form $x_1x_2\dotsb x_l$, where each $x_i$ is either $1$ or $2$.
  Then
  \begin{equation}
    \label{eq:2}
    f_w = \prod_{\{i\mid x_i = 2\}} \left(\sum_{j=i}^l x_j -1\right).
  \end{equation}
\end{lemma}
\begin{proof}
  We prove the result by induction on the rank of $w$.
  If $r(w)=0$, the result is obviously true.
  If $w=1v$ for some $v$, then $w^-$ consists of the single element $v$, and so $f_w=f_v$.
  Since $v$ and $w$ differ only by a leftmost $1$, the right hand side of (\ref{eq:2}) is the same for both $v$ and $w$, so the Lemma follows for $w$.

  On the other hand, if $w = 2v$, then we wish to show that
  \begin{displaymath}
    f_w = (r(v)+1)f_v.
  \end{displaymath}
  We have:
  \begin{displaymath}
    w^- = \{1v\}\coprod \{2u\mid u\in v^-\},
  \end{displaymath}
  so that
  \begin{align*}
    f_w & = f_{1v} + \sum_{u\in v^-} f_{2u}\\
    & = f_v + \sum_{u\in v^-} [r(u) + 1]f_u & \text{[by induction]}\\
    & = f_v + [r(u) + 1] f_v\\
    & = [r(v) + 1] f_v, &\text{[since $r(u) + 1 = r(v)$]}
  \end{align*}
  as required.
\end{proof}
\begin{definition}
  [Odd word]
  We say that $w\in F(n)$ is an odd word if $f_w$ is odd.
  The subgraph induced in $F(n)$ by the set of odd words is denoted by $F_\odd(n)$.
\end{definition}
Thus $w$ is an odd word if and only if each of the factors on the right hand side of (\ref{eq:2}) is odd, which means that the number of occurrences of $1$ to the right of each $2$ in $w$ is even.
This may be restated in the following form:
\begin{corollary}
  \label{corollary:odd-words}
  A word $w\in F$ is odd if and only if it is of the form $a_0\dotsc a_{k-1}$ or $1a_0\dotsc a_{k-1}$, where each $a_i$ is either $2$ or $11$.
\end{corollary}
Note that, in the above result, $k=\lfloor r(w)/2\rfloor$.

When $w$ is as in the above corollary, $f_w$ can be computed easily (it is convenient to reverse the order of the indices):
\begin{corollary}
  \label{lemma:f-value}
  Let $w$ be of the form $a_{k-1}\dotsb a_1a_0$ or $a_{k-1}\dotsb a_1a_0$, where $a_i$ is either $11$ or $2$ for each $i$.
  Then
  \begin{displaymath}
    f_w = (2k_1 + 1)(2k_2 + 1) \dotsb (2k_r + 1),
  \end{displaymath}
  where $\{k_1,\dotsc,k_r\}$ is the set of those values of $i$ for which $a_i=2$.
\end{corollary}
As an immediate consequence of Corollary~\ref{corollary:odd-words}, we get:
\begin{corollary}
  The number of odd words in $F(n)$ is $2^{\lfloor n/2\rfloor}$.
\end{corollary}
\section{The Macdonald Tree in the Young-Fibonacci Graph}
The subgraph in Young's graph induced by odd partitions is a binary tree.
This tree was called the Macdonald tree in \cite{APS}.
Figure~\ref{yf} indicates that an analogous result holds for the Young-Fibonacci graph.
The edges joining odd words are rendered in green, while the remaining edges are rendered in orange.
The green edges in Figure~\ref{yf} clearly form a binary tree.
The following theorem (see also \cite[Theorem~5]{APS}) is a direct consequence of Corollary~\ref{lemma:f-value}:
\begin{theorem}
  [Macdonald tree of the Young-Fibonacci graph]
  \label{theorem:macdonald-tree}
  If $n$ is even, and $w\in F_\odd(n)$, then $w^+\cap F_\odd(n+1) = \{1w\}$.
  If $n$ is odd, and $w\in F_\odd(n)$, then $w=1v$ for a unique element $v\in F_\odd(n-1)$ and $w^+\cap F_\odd(n+1) = \{11v, 2v\}$.

  Consequently, the subgraph $F_\odd$ induced in the Young-Fibonacci graph by odd words is a binary tree, where every node of even rank has one child, and every node of odd rank has two children.
  We call this tree the Macdonald tree of the Young-Fibonacci graph.
\end{theorem}
The first seven rows of the Macdonald tree of the Young-Fibonacci graph are shown in Figure~\ref{fig:mtyf}.
\begin{figure}
  \begin{displaymath} 
\resizebox{\textwidth}{!}{
        \xymatrix@R=2mm@C=2mm{ 
        \text{\tiny$1111111$}\ar@{-}[dd] &   & \text{\tiny$121111$}\ar@{-}[dd] & \text{\tiny$111211$}\ar@{-}[dd] & & \text{\tiny$12211$}\ar@{-}[dd] &  \text{\tiny$111112$} \ar@{-}[dd] & & \text{\tiny$12112$}\ar@{-}[dd]& \text{\tiny$11122$}\ar@{-}[dd] & &  \text{\tiny$1222$}\ar@{-}[dd]  \\
        &&&&&&&&&&&\\
       \text{\tiny$111111$} \ar@{-}[ddr] & & \text{\tiny$21111$} \ar@{-}[ddl] &  \text{\tiny$11211$} \ar@{-}[ddr] & & \text{\tiny$2211$} \ar@{-}[ddl] & \text{\tiny$11112$} \ar@{-}[ddr] & & \text{\tiny$2112$} \ar@{-}[ddl] & \text{\tiny$1122$} \ar@{-}[ddr]& & \text{\tiny$222$} \ar@{-}[ddl]  \\ 
        &&&&&&&&&&&\\
         & \text{\tiny$11111$}\ar@{-}[dd]  & & & \text{\tiny$1211$}\ar@{-}[dd] & & & \text{\tiny$1112$}\ar@{-}[dd] & & &  \text{\tiny$122$}\ar@{-}[dd] &  \\
           &&&&&&&&&&&\\
       & \text{\tiny$1111$}\ar@{-}[ddrr] & & & \text{\tiny$211$}\ar@{-}[ddl] &  & & \text{\tiny$112$}\ar@{-}[ddrr] & & & \text{\tiny$22$}\ar@{-}[ddl] \\
         &&&&&&&&&&&\\
         & & & \text{\tiny$111$}\ar@{-}[dd]  & & & & & & \text{\tiny$12$} \ar@{-}[dd]& & &\\
           &&&&&&&&&&&\\
         & & & \text{\tiny$11$}\ar@{-}[ddrrr] & &  & & & & \text{\tiny$2$} \ar@{-}[ddlll]& & &\\
           &&&&&&&&&&&\\
         & & & & & & \text{\tiny$1$} \ar@{-}[dd]   & & & &\\   
           &&&&&&&&&&&\\
         & & & & & & \text{\tiny$\emptyset$} & & & & &}
}
    \end{displaymath}
    \caption{Macdonald tree of the Young-Fibonacci graph}
    \label{fig:mtyf}
\end{figure}  
In comparison with the recursive structure of the Macdonald tree in Young's graph \cite[Section~4]{APS}, the recursive structure of the Macdonald tree in the Young-Fibonacci graph is very simple.
Let $F_\odd^w$ denote the induced subtree, rooted at $w$, consisting only of those nodes which have $w$ as an ancestor.
In other words, these are the nodes corresponding to elements $v\in F_\odd$ such that $v\geq w$ in the partial order, and $v$ is connected to $w$ by a saturated chain of $F$ that is contained in $F_\odd$. 
For any word $w\in F_\odd$, let $F_\odd w$ denote a copy of $F_\odd$, where each vertex label $v\in F_\odd$ is replaced by $vw$.
\begin{theorem}
  [Recursive Definition of the Macdonald Tree of the Young-Fibonacci Lattice]
  \label{sec:macdonald-tree-recursive}
  \begin{displaymath}
    F_\odd^w = 
    \vcenter{\vbox{
        \xymatrix{
           F_\odd 11w\ar@{-}[dr] & & F_\odd 2w\ar@{-}[dl] \\
           &1w \ar@{-}[d]&\\
           &w&}
      }}
  \end{displaymath}
\end{theorem}
In order to prove the equidistribution of $f$-numbers of elements of $F_\odd(n)$, we replace the vertex labels in $F_\odd$ by their $f$-numbers:
\begin{definition}
  [$f$-valued Macdonald Tree]
  The $f$-valued Macdonald tree $\F_\odd$ of the Young-Fibonacci lattice is obtained from $F_\odd$ by replacing the label of each node $w\in F_\odd$ by $f_w$.
\end{definition}
Denote by $\F_\odd^w$ the induced subtree of the $f$-valued Macdonald tree with nodes consisting only of those nodes which have $w$ as an ancestor.
For any scalar $u$, let $u\F_\odd$ denote the tree with labelled nodes whose underlying unlabelled tree is $F_\odd$ and the node corresponding to $w$ is labelled by $uf_w$.
\begin{theorem}
  [Recursive Description of $f$-valued Macdonald Tree]
  \label{theorem:f-val}
  For each $w\in F_\odd(2m)$, we have:
  \begin{displaymath}
    \F_\odd^w = 
    \vcenter{\vbox{
        \xymatrix{
        f_w\F_\odd\ar@{-}[dr] & & (2m+1)f_w\F_\odd\ar@{-}[dl] \\
         &f_w\ar@{-}[d]& \\
          &f_w&   
               }}}
  \end{displaymath}
\end{theorem}
\begin{proof}
  This theorem is a consequence of Corollary~\ref{lemma:f-value} and Theorem~\ref{sec:macdonald-tree-recursive}.
\end{proof}
\begin{figure}
  \label{fvalue}
 \begin{displaymath}
\resizebox{\textwidth}{!}{
         \xymatrix@R=2mm@C=2mm{ 
       \text{\tiny$1$} \ar@{-}[dd]&   & \text{\tiny$1 \times 5$} \ar@{-}[dd] & \text{\tiny$1 \times 3$} \ar@{-}[dd] & & \text{\tiny$1 \times 3 \times 5$} \ar@{-}[dd] & \text{\tiny$1$}  \ar@{-}[dd] & & \text{\tiny$1 \times 5$}\ar@{-}[dd]& \text{\tiny$1 \times 3$}\ar@{-}[dd] & & \text{\tiny$1 \times 3 \times 5$} \ar@{-}[dd]  \\ 
         &&&&&&&&&&&\\
       \text{\tiny$1$} \ar@{-}[ddr]& &\text{\tiny$1\times 5$} \ar@{-}[ddl] &  \text{\tiny$1 \times 3$} \ar@{-}[ddr] & & \text{\tiny$1 \times 3 \times 5$} \ar@{-}[ddl] & \text{\tiny$1$} \ar@{-}[ddr] & & \text{\tiny$1 \times 5$} \ar@{-}[ddl] & \text{\tiny$1 \times 3$} \ar@{-}[ddr]& & \text{\tiny$1 \times 3 \times 5$} \ar@{-}[ddl]  \\ 
         &&&&&&&&&&&\\
         & \text{\tiny$1$}\ar@{-}[dd]  & & & \text{\tiny$1 \times 3$}\ar@{-}[dd] & & & \text{\tiny$1$}\ar@{-}[dd] & & & \text{\tiny$1 \times 3$}\ar@{-}[dd]& \\
           &&&&&&&&&&&\\
         & \text{\tiny$1$}\ar@{-}[ddrr] & & & \text{\tiny$1 \times 3$}\ar@{-}[ddl] &  & & \text{\tiny$1$}\ar@{-}[ddrr] & & & \text{\tiny$1 \times 3$}\ar@{-}[ddl] \\
           &&&&&&&&&&&\\
         & & & \text{\tiny$1$}\ar@{-}[dd] & & & & & & \text{\tiny$1$}\ar@{-}[dd]& & &\\
           &&&&&&&&&&&\\
         & & & \text{\tiny$1$}\ar@{-}[ddrrr] & &  & & & & \text{\tiny$1$} \ar@{-}[ddlll]& & &\\
           &&&&&&&&&&&\\
         & & & & & & \text{\tiny$1$} \ar@{-}[dd]   & & & &\\   
           &&&&&&&&&&&\\
         & & & & & & \text{\tiny$1$} & & & & &}
}
\end{displaymath}
 \caption{$f$-valued Macdonald Tree}
\end{figure}
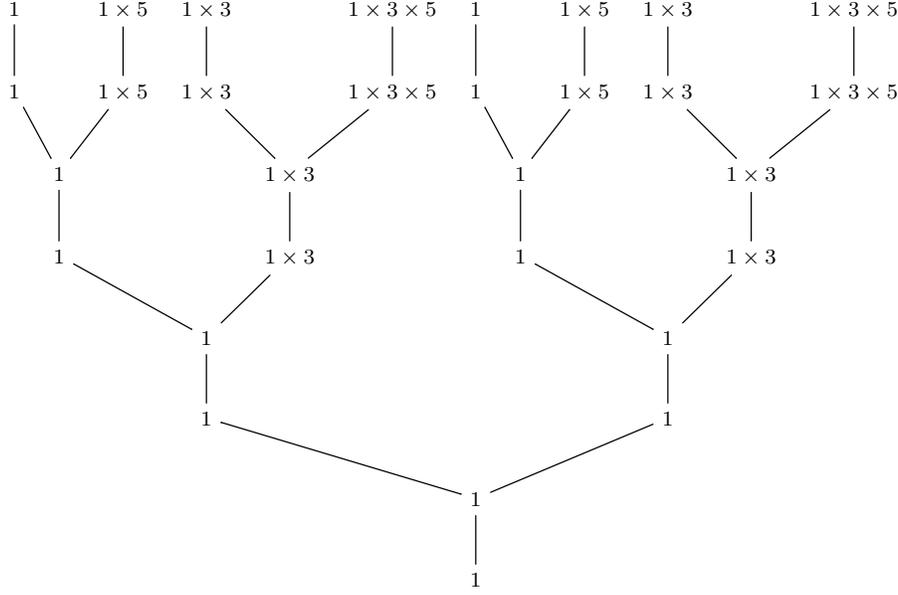
\section{Equidistribution of Residues}
We say that the residues modulo $2^k$ of
\begin{displaymath}
  \{f_w \mid w\in F_\odd(n)\}
\end{displaymath}
are equidistributed if the numbers
\begin{displaymath}
  m_i = |\{w\in F_\odd(n)\mid f_w \equiv i \mod 2^k\}|
\end{displaymath}
are the same for all odd integers $1\leq i \leq 2^k-1$.
In this section, we show that, for every positive integer $k$, the residues modulo $2^k$ of $\{f_w\mid w\in F_\odd(n)\}$ are equidistributed for sufficiently large $n$.
\begin{lemma}
  \label{corollary:one-step}
  Suppose the residues modulo $2^k$ of $\{f_w\mid w\in F_\odd(n)\}$ are equidistributed, then so are the residues modulo $2^k$ of $\{f_w\mid w\in F_\odd(n+1)\}$.
\end{lemma}
\begin{proof}
  If $n$ is even, then the $f$-numbers in $F(n)$ and $F(n+1)$ remain the same.
  If $n = 2m+1$ is odd then,
  as multisets,
  \begin{multline*}
    \{f_w \mod 2^k\mid w\in F_\odd(n+1)\} = \{f_w\mod 2^k \mid w \in F_\odd(n)\} \\
    \coprod \{(2m+1)f_w\mid w\in F_\odd(n)\}.
  \end{multline*}
  Since $2m+1$ is a unit in $\mathbf Z/2^k\mathbf Z$, multiplication by it is a permutation of the multiplicative group $(\mathbf Z/2^k\mathbf Z)^*$.
  It follows that the residues in the second subset coincide with those in the first one, which are equidistributed, and so the corollary holds.
\end{proof}
Next we shall determine, for each $k$, a value of $n$ for which the residues $\{f_w \mod 2^k\mid w\in F_\odd(n)\}$ are equidistributed.
\begin{theorem}
  [Main Theorem]
  \label{theorem:main}
  If $n\geq 2^{k-1}+2$, then the residues
  \begin{equation}
    \label{eq:3}
    \{f_w \mod 2^k \mid w \in F_\odd(n)\}
  \end{equation}
  are equidistributed in $(\mathbf Z/2^k\mathbf Z)^*$.
\end{theorem}
\begin{proof}
  The proof proceeds by induction on $k$.
  The result holds trivially for $k=1$.

  Take $k>1$.
  By Lemma~\ref{corollary:one-step}, it suffices to show that the residues in~(\ref{eq:3}) equidistribute for $n=2^{k-1} +2$.
  By the induction hypothesis and Lemma~\ref{corollary:one-step}, we know that the residues 
  \begin{equation}
    \label{eq:4}
    \{f_w \mod 2^{k-1} \mid w \in F_\odd(2^{k-1}+2)\}
  \end{equation}
  equidistribute in $(\mathbf Z/2^{k-1}\mathbf Z)^*$.
  Therefore it suffices to show that, for each odd integer $0<i<2^{k-1}$,
  the number of occurrences of $i$ modulo $2^k$ is the same as the number of occurrences of $2^{k-1} + i$ modulo $2^k$ in (\ref{eq:4}).
  Observe that
  \begin{displaymath}
    2^{k-1} + i \equiv (2^{k-1} + 1)i \mod 2^k.
  \end{displaymath}
  By Theorem~\ref{theorem:f-val}, for each $w\in F_\odd(2^{k-1}+1)$, the two odd words in $w^+$ have residues $f_w$ and $(2^{k-1} + 1)f_w$.
  If one  of these residues modulo $2^k$ is $i$ for some $0<i<2^{k-1}$, then the other is $(2^k + i)$ modulo $2^k$.
  This concludes the proof of the theorem.
\end{proof}
In view of Corollary~\ref{lemma:f-value}, our main theorem about the residues modulo $2^k$ of the $f$-numbers of elements of $F_\odd(n)$ has a purely number-theoretic formulation:
\begin{corollary}
  \label{corollary:number-theory}
  Let $n$ and $k$ be positive integers.
  Consider the multiset of products of odd integers:
  \begin{displaymath}
    \Pi_n = \{x_1x_2\dotsb x_t\mid 1\leq x_1 <\dotsb < x_t\leq n, \: x_j \text{ odd for } 1\leq j \leq t\}.
  \end{displaymath}
  (The empty product, which is taken to be $1$, is included, so $\Pi_n$ has $2^{\lfloor n/2\rfloor}$ elements.)
  For $i=1,3,5,\dots,2^k-1$, let $m_i$ denote the number (counted with multiplicity) of elements of $\Pi_n$ that are congruent to $i$ modulo $2^k$.
  If $n\geq 2^{k-1} + 2$, then all the numbers $m_i$ are equal.
\end{corollary}
\begin{proof}
  By corollaries~\ref{corollary:odd-words} and~\ref{lemma:f-value}, the multiset $\Pi_n$ coincides with the labels in the $n$th row of the $f$-valued Macdonald tree, so the corollary follows from Theorem~\ref{theorem:main}.  
\end{proof}

\section{A note on other primes}
\label{sec:others}

The following generalization of Corollary~\ref{corollary:odd-words} holds for any prime number $p$:
\begin{theorem}
  Let $p$ be a prime number.
  A word $w$ in $F$ has $f_w$ coprime to $p$ if and only if $w$ is of the form $a_0a_1\dotsb a_k$, where $a_0$ is of rank less than $p$, and $a_i$ has rank $p$ for $i=1,\dotsc,k$.
\end{theorem}
As a result, we have:
\begin{theorem}
  Let $p$ be any prime number, and let $n$ be a positive integer of the form $n = pm +r$, where $0\leq r <p$.
  The number $C_p(n)$ of words $w\in F(n)$ for which $f_w$ is coprime to $p$ satisfies the identity:
  \begin{displaymath}
    C_p(n) = C_p(p)^m C_p(r).
  \end{displaymath}
  Thus the sequence $\{C_p(n)\}_{n=1}^\infty$ is determined by its first $p$ values.
\end{theorem}
However residues modulo primes $p>2$ do not equidistribute.
\subsection*{Acknowledgements}
We thank Arvind Ayyer, R. Balasubramanian, Kamalakshya Mahatab, Anirban Mukhopadhyay, Parameswaran Sankaran, and Steven Spallone for many helpful discussions. The second author was supported by a Swarnajayanti Fellowship of the Department of Science \& Technology (India). The third author thanks the Institute of Mathematical Sciences, Chennai, for offering him a summer internship during which this project was begun.
\bibliographystyle{abbrv}
\bibliography{refs}
\end{document}